\documentclass[article,12pt]{amsart}

\usepackage{amsfonts}
\usepackage{amsmath}
\usepackage{amssymb}
\usepackage{graphicx}
\usepackage{epsfig}
\usepackage{enumitem}
\usepackage{fullpage}
\usepackage{color}
\usepackage{bbm}

\allowdisplaybreaks

\numberwithin{equation}{section}

\theoremstyle{plain}
\newtheorem{thm}{Theorem}[section]
\newtheorem{prop}{Proposition}[section]

\newtheorem{lem}{Lemma}[section]

\theoremstyle{remark}

\newtheorem{rem}{Remark}[section]

\newcommand{\dE}{\mathbb{E}}

\newcommand{\dR}{\mathbb{R}}
\newcommand{\dZ}{\mathbb{Z}}

\newcommand{\ind}{\mathbbm{1}}

\newcommand{\cD}{\mathcal{D}}
\newcommand{\cF}{\mathcal{F}}
\newcommand{\cH}{\mathcal{H}}
\newcommand{\cN}{\mathcal{N}}

\newcommand{\de}{\mathrm{e}}

\newcommand{\veps}{\varepsilon}
\newcommand{\hsp}{\hspace{0.5cm}}
\newcommand{\wh}{\widehat}
\newcommand{\wht}{\wh{\theta}_{n}}
\renewcommand{\vec}{\textnormal{vec}}
\renewcommand{\geq}{\geqslant}
\renewcommand{\leq}{\leqslant}

\newcommand{\cvgas}{~ \overset{\textnormal{a.s.}}{\longrightarrow} ~}

\newcommand{\cvgl}{~ \overset{\cD}{\longrightarrow} ~}

\newcommand{\udots}{\mathinner{\mskip1mu\raise-2pt\vbox{\kern7pt\hbox{.}}\mskip2mu\raise1pt\hbox{.}\mskip2mu\raise4pt\hbox{.}\mskip1mu}}
  
\email{frederic.proia@univ-angers.fr}
\email{marius.soltane.etu@univ-lemans.fr}


\begin{document}

\title[RCAR process with correlated coefficients]
{Comments on the presence of serial correlation in the random coefficients of an autoregressive process
\vspace{2ex}}
\author[F. Pro\"ia]{Fr\'ed\'eric Pro\"ia}
\address{Laboratoire Angevin de REcherche en MAth\'ematiques (LAREMA), CNRS, Universit\'e d'Angers, Universit\'e Bretagne Loire. 2 Boulevard Lavoisier, 49045 Angers cedex 01, France.}
\author[M. Soltane]{Marius Soltane}
\address{Laboratoire Manceau de Math\'ematiques, Le Mans Universit\'e, Avenue O. Messiaen, 72085 Le Mans cedex 9, France.}

\thanks{}

\begin{abstract}
We consider an RCAR$(p)$ process and we establish that the standard estimation lacks consistency as soon as there exists a nonzero serial correlation in the coefficients. We give the correct asymptotic behavior and some simulations come to illustrate the results.
\end{abstract}

\maketitle

\section{Introduction and Motivations}
\label{SecIntro}

This note is devoted to the estimation issue in the random coefficients non-linear generalization of the autoregressive process or order $p$, that is usually defined as
\begin{equation*}
X_{t} = \sum_{k=1}^{p} (\theta_k + b_{k,\, t})\, X_{t-k} + \veps_{t}
\end{equation*}
where, for each $1 \leq k \leq p$, $(b_{k,\, t})_{t\, \in\, \dZ}$ is a white noise included in the coefficients, independent of $(\veps_{t})_{t\, \in\, \dZ}$. We refer the reader to the monograph of \cite{NichollsQuinn82} for a rich introduction to the topic. The conditions of existence of a stationary solution to this equation have been widely studied since the seminal works of \cite{Andel76} or \cite{NichollsQuinn81a}, see also the more recent paper of \cite{AueHorvathSteinebach06}. The estimation of $\theta$ has been extensively developed to this day as well. \cite{NichollsQuinn81b} suggest to make use of the OLS which turns out to be the same, whether there is randomness or not in the coefficients. Under the stationarity conditions, the OLS is known to be strongly consistent and asymptotically normal in both cases (but with different variances). Looking for a unified theory, that is, irrespective of stationarity issues, \cite{AueHorvath11} and \cite{BerkesHorvathLing09} show that the QMLE is consistent and asymptotically normal as soon as there is some randomness in the coefficients, the variability of which allows to circumvent the well-known unit root issues. Later, \cite{HillPeng14} develop an empirical likelihood estimation which is asymptotically normal even when the coefficients are non-random. Let us also mention the WLS approach of \cite{Schick96} and the M-estimation of \cite{KoulSchick96} adapted to stationary processes, which provide asymptotically normal estimators with optimal variance. Yet, the OLS estimation, easy to compute, that does not require numerical optimization or additional parametrization, is still very popular. We shall note that all these works, except those of Nicholls and Quinn, are related to the first-order process. By contrast, studies on general multivariate RCAR processes do not seem widespread in the literature, to the best of our knowledge. In this paper, we are interested in the implications of serial correlation in the random coefficients of a $p$-order RCAR process. Among the fundamental hypotheses of the usual theory, the coefficients are considered as white noises. However, our main statement is that, in a chronological model, it is unlikely that the coefficients are serially uncorrelated, from the moment they are considered as random. From that point of view, we intend to show that the OLS method may lead to inappropriate conclusions when there is some serial correlation in the random coefficients. This work can be seen as a partial generalization of \cite{ProiaSoltane18} where in particular, for the first-order RCAR process, the lack of consistency is established together with the correct behavior of the statistic. For this purpose, consider the RCAR process $(X_{t})_{t\, \in\, \dZ}$ of order $p$ generated by the autoregression
\begin{equation}
\label{RCAR}
X_{t} = \sum_{k=1}^{p} (\theta_k + \alpha_k\, \eta_{k,\, t-1} + \eta_{k,\, t})\, X_{t-k} + \veps_{t}
\end{equation}
where $(\veps_{t})_{t\, \in\, \dZ}$ and $(\eta_{k,\, t})_{t\, \in\, \dZ}$ are uncorrelated strong white noises (that is, i.i.d. sequences) with $q$th-order moments $\sigma_{q}$ and $\tau_{k,\, q}$ respectively. We also assume that $(\eta_{k,\, t})_{t\, \in\, \dZ}$ and $(\eta_{\ell,\, t})_{t\, \in\, \dZ}$ are uncorrelated, for $k \neq \ell$. The main fact to give credit to this model is Prop. 3.2.1 of \cite{BrockwellDavis91} which states that any stationary process with finite memory can be expressed by a moving average structure. Thus, by means of an MA(1) model, we introduce a lag-one memory in the random coefficients as a simple case study meant to illustrate the effect on the estimation. Note that, although much more cumbersome, the same study may lead to the same conclusions for any stationary coefficients with finite memory, \textit{via} general MA structures. However, for infinite memory (\textit{e.g.} stationary random AR(1) coefficients), the problem is probably much more complicated. In Section \ref{SecRes}, we introduce our hypotheses and we detail the asymptotic behavior of the standard OLS in presence of serial correlation, under the stationarity conditions. In particular, we brought to light a disturbing consequence when testing the significance of $\theta$. Some simulations come to illustrate our results in Section \ref{SecAppli} while Section \ref{SecPro} contains the proofs. Finally, we postpone to the Appendix the purely computational steps, for the sake or readability.

\section{Influence of serial correlation in the coefficients}
\label{SecRes}

It will be convenient to write \eqref{RCAR} in the vector form
\begin{equation}
\label{RCARVec}
\Phi_{t} = (C_{\theta} + N_{t-1}\, D_{\alpha} + N_{t})\, \Phi_{t-1} + E_{t}
\end{equation}
where $\Phi_{t}^{\, T} = (X_{t}, \hdots, X_{t-p+1})$, $E_{t}^{\, T} = (\veps_{t}, 0, \hdots, 0)$, $N_{t}$ is a matrix with $\eta_t = (\eta_{1,\, t}, \hdots, \eta_{p,\, t})$ in the first row and 0 elsewhere, $D_{\alpha} = \textnormal{diag}(\alpha_1, \hdots, \alpha_p)$ and $C_{\theta}$ is the companion matrix of the underlying AR$(p)$ process. We already make the assumption that any odd moment of $\veps_{0}$ and $\eta_{k,\, 0}$ is zero as soon as it exists ($\sigma_{2q+1} = \tau_{k,\, 2q+1} = 0$), and that the distribution of the noises guarantees
\begin{equation}
\label{HypLog}
\dE[\ln \Vert C_{\theta} + N_0\, D_{\alpha} + N_1 \Vert] < 0 \hsp \text{and} \hsp \dE[\ln^{+} \vert \veps_0 \vert] < +\infty.
\end{equation}
Those moments conditions are assumed to hold throughout the study. As for the moments of the process, the hypotheses are related to the space $\Theta$ where the parameters live. Define
\begin{equation*}
\Theta_m = \{ \{ \theta, \alpha, \sigma_q, \tau_{k,\, q^{\prime}} \} ~  \vert ~ \dE[\eta_{k,\, t}^{\, a}\, X_{t}^{\, m}] < +\infty \text{ for } a \in \{ 0, \hdots, m \} \text{ and } k \in \{ 1, \hdots, p \} \}
\end{equation*}
for $m=2, 4$, where $\sigma_q$ and $\tau_{k,\, q^{\prime}}$ name in a generic way the highest-order moments of $(\veps_{t})$ and $(\eta_{k,\, t})$ necessary to obtain $\dE[\eta_{k,\, t}^{\, a}\, X_{t}^{\, m}] < +\infty$. We give in the Appendix some precise facts about $\Theta_2$ and $\Theta_4$, in particular we can see that $q = m$ and $q^{\, \prime} = 2m$ are needed. The pathological cases are put together in a set called $\Theta^{*}$ that is built step by step during the reasonings. First, we have the following causal representation showing in particular that the process is adapted to the filtration
\begin{equation}
\label{Filt}
\cF_{t} = \sigma( (\veps_{s}, \eta_{1,\, s}, \hdots, \eta_{p,\, s}),\, s \leq t).
\end{equation}
The reader will find the proofs of our results in Section \ref{SecPro}.

\begin{prop}
\label{PropErgo}
For all $t \in \dZ$,
\begin{equation}
\label{ExprCausal}
\Phi_{t} = E_{t} + \sum_{k=1}^{\infty} E_{t-k}\, \prod_{\ell=0}^{k-1} (C_{\theta} + N_{t-\ell-1}\, D_{\alpha} + N_{t-\ell}) \hsp \textnormal{a.s.}
\end{equation}
In consequence, $(X_{t})$ is strictly stationary and ergodic.
\end{prop}
Suppose now that $(X_{-p+1}, \hdots, X_{n})$ is an available trajectory from which we deduce the standard OLS estimator of $\theta$, the mean value of the coefficients,
\begin{equation}
\label{OLS}
\wht = S_{n-1}^{\, -1} \sum_{t=1}^{n} \Phi_{t-1}\, X_{t} \hsp \text{where} \hsp S_{n} = \sum_{t=0}^{n} \Phi_{t}\, \Phi_{t}^{\, T}.
\end{equation}
To simplify, the initial vector $\Phi_0$ is assumed to follow the stationary distribution of the process. It is important to note that \eqref{OLS} is the OLS of $\theta$ with respect to $\sum_{t} (X_{t} - \theta^{\, T} \Phi_{t-1})^2$ even when $\alpha \neq 0$, as it is done in Sec. 3 of \cite{NichollsQuinn81b}, because our interest is precisely to show that the usual estimation (\textit{i.e.} assuming $\alpha=0$) may lead to inappropriate conclusions in case of misspecification of the coefficients.

\begin{rem}
\label{RemErgo}
From the causal representation, ergodicity implies that
\begin{equation}
\label{Ergo}
\frac{1}{n} \sum_{t=1}^{n} Y_{t} \cvgas \dE[Y_0] < +\infty \hsp \text{for any process of the form} \hsp Y_{t} = \prod_{i=1}^{m} X_{t-d_{i}}\, \eta_{k_i,\, t-d_i^{\, \prime}}^{\, a_i}
\end{equation}
where $0 \leq d_i^{\, \prime} \leq d_i$, $a_i \in \{0, 1\}$ and $1 \leq k_i \leq p$, provided that $\Theta \subset \Theta_m$ ($m = 2,4$). In all the study, any term taking the form of $Y_{t}$ is called a \textit{second-order} (or \textit{fourth-order}) \textit{isolated term} when $m=2$ (or $m=4$). According to \eqref{Ergo}, isolated terms satisfy $Y_{t} = o(n)$ and thus $\vert Y_{t} \vert^{1/2} = o(\sqrt{n})$ a.s. Such arguments are frequently used in the proofs.
\end{rem}
The autocovariances of the stationary process are denoted by $\ell_i = \dE[X_{i}\, X_0]$ and, by ergodicity, the sample covariances are strongly consistent estimators, that is
\begin{equation}
\label{SampleCov}
\frac{1}{n} \sum_{t=1}^{n} X_{t}\, X_{t-i} \cvgas \ell_i.
\end{equation}
Based on the autocovariances, we build
\begin{equation}
\label{Lam}
\Lambda_0 = \begin{pmatrix}
\ell_0 & \ell_1 & \cdots & \ell_{p-1} \\
\ell_1 & \ell_0 & \cdots & \ell_{p-2} \\
\vdots & \vdots & & \vdots \\
\ell_{p-1} & \ell_{p-2} & \cdots & \ell_0
\end{pmatrix} \hsp \text{and} \hsp L_1 = \begin{pmatrix}
\ell_1 \\
\ell_2 \\
\vdots \\
\ell_p
\end{pmatrix}.
\end{equation}
The matrix $\Lambda_0$ is clearly positive semi-definite but the case where it would be non-invertible is part of $\Theta^{*}$. This is the multivariate extension of $2\, \alpha\, \tau_2 \neq 1$ in \cite{ProiaSoltane18}, this can also be compared to assumption (v) in \cite{NichollsQuinn81b}. The asymptotic behavior of \eqref{OLS} is now going to be studied in terms of convergence, normality and rate. The result below is immediate from the ergodic theorem, provided that $\Theta \subset \Theta_2$.

\begin{thm}
\label{ThmCvgEst}
Assume that $\Theta \subset \Theta_2 \backslash \Theta^{*}$. Then, we have the almost sure convergence
\begin{equation*}
\wht \cvgas \theta^{\, *} = \Lambda_0^{-1}\, L_1.
\end{equation*}
\end{thm}
We give in Remark \ref{RemEstParam} of the Appendix another expression of $\theta^{\, *}$ which directly shows that $\theta^{\, *} = \theta$ as soon as $\alpha=0$, as it is established in Thm. 4.1 of \cite{NichollsQuinn81b}. However, except for $p=1$ (see \cite{ProiaSoltane18}), we generally do not have $\theta^{\, *} = 0$ when $\theta=0$, and we will use this disturbing fact in the short example of Section \ref{SecAppli}.

\begin{thm}
\label{ThmTLCEst}
Assume that $\Theta \subset \Theta_4 \backslash \Theta^{*}$. Then, there exists a limit matrix $L$ such that we have the asymptotic normality
\begin{equation*}
\sqrt{n}\, (\wht - \theta^{\, *}) \cvgl \cN(0,\, \Lambda_0^{\, -1}\, L\, \Lambda_0^{\, -1}).
\end{equation*}
\end{thm}

\begin{thm}
\label{ThmRatEst}
Assume that $\Theta \subset \Theta_4 \backslash \Theta^{*}$. Then, we have the rate of convergence
\begin{equation*}
\limsup_{n\, \rightarrow\, +\infty} ~ \frac{n}{2\, \ln \ln n}\, \Vert \wht - \theta^{\, *} \Vert^{\, 2} < +\infty \hsp \textnormal{a.s.}
\end{equation*}
\end{thm}
In particular, we can see that despite the correlation in the noise of the coefficients, the hypotheses are sufficient to ensure that the estimator achieves the usual rate of convergence in stable autoregressions, \textit{i.e.}
\begin{equation*}
\Vert \wht - \theta^{\, *} \Vert = O\bigg( \sqrt{\frac{\ln \ln n}{n}} \bigg) \hsp \textnormal{a.s.}
\end{equation*}
Note that $\sigma_2$ and $\tau_{k,\, 4}$ are involved in $\Theta_2$ to obtain the almost sure convergence, whereas $\sigma_4$ and $\tau_{k,\, 8}$ are involved in $\Theta_4$ to get the asymptotic normality and the rate of convergence. 

\begin{rem}
\label{RemHyp}
Assuming that $(\eta_{k,\, t})$ and $(\eta_{\ell,\, t})$ are uncorrelated for $k \neq \ell$ is a matter of simplification of the calculations. We might as well consider a covariance structure for $(\eta_{1,\, t}, \hdots, \eta_{p,\, t})$ and obtain the same kind of results, at the cost of additional parameters and refined hypotheses (see \textit{e.g.} assumption (iii) in \cite{NichollsQuinn81b}).
\end{rem}

\section{Illustration and Perspectives}
\label{SecAppli}

To conclude this short note, let us illustrate a consequence of the presence of correlation in the coefficients when testing the significance of $\theta$. For the sake of simplicity, take $p=2$ and consider the test of $\cH_0 : ``\theta_2 = 0"$ against $\cH_1 : ``\theta_2 \neq 0"$. Thanks to Remark \ref{RemEstParam}, it is possible to show that, under $\cH_0$,
\begin{equation*}
\theta_1^{\, *} = \frac{\theta_1\, (\theta_1^{\, 2} + \beta_2 - (1-\beta_1)^{\, 2} - \beta_1\, ( \beta_1 + \beta_2 - 1))}{(\theta_1 - 2\, \beta_1 - \beta_2 + 1)\, (\theta_1 + 2\, \beta_1 + \beta_2 - 1)}
\end{equation*}
and
\begin{equation*}
\theta_2^{\, *} = \frac{\theta_1^{\, 2}\, \beta_2 - \beta_1\, ((1-\beta_2)^{\, 2} + 4\, \beta_1\, ( \beta_1 + \beta_2 - 1))}{(\theta_1 - 2\, \beta_1 - \beta_2 + 1)\, (\theta_1 + 2\, \beta_1 + \beta_2 - 1)}
\end{equation*}
where $\beta_1 = \alpha_1\, \tau_{1,\, 2}$ and $\beta_2 = \alpha_2\, \tau_{2,\, 2}$. As a consequence, $\sqrt{n}\, \vert \wh{\theta}_{2,\, n} \vert$ is almost surely divergent and even if $\theta_2=0$, we may detect a non-zero-mean coefficient where there is in fact a zero-mean autocorrelated one. Worse, suppose also that $\alpha_2 = \tau_{2,\, 2} = 0$ so that there is no direct influence of $X_{t-2}$ on $X_{t}$. Then, we still generally have $\theta_2^{\, *} \neq 0$. In other words, a correlation in the random coefficient associated with $X_{t-1}$, \textit{i.e.} $\alpha_1 \neq 0$, generates a spurious detection of a direct influence of $X_{t-2}$ on $X_{t}$. This phenomenon can be observed on some simple but representative examples. To test $\cH_0$, we infer two statistics from Thm. 4.1 of \cite{NichollsQuinn81b} and the procedure of estimation given by the authors,
\begin{equation*}
Z_{1,\, n} = \frac{\sqrt{n}\, \wh{\theta}_{2,\, n}}{v_{1,\, n}} \hsp \text{and} \hsp Z_{2,\, n} = \frac{\sqrt{n}\, \wh{\theta}_{2,\, n}}{v_{2,\, n}}.
\end{equation*}
The first one takes into account random coefficients (which means that $\tau_{1,\, 2}$ and $\tau_{2,\, 2}$ are estimated to get $v_{1,\, n}$) and the second one is built assuming fixed coefficients ($\tau_{1,\, 2}$ and $ \tau_{2,\, 2}$ are not estimated but set to 0). Note that we make sure that $\Theta \subset \Theta_4$ when we vary the settings. Unsurprisingly, $Z_{2,\, n}$ is less reliable since it does not model the random effect owing to $\tau_{1,\, 2} > 0$ and the corrected statistic $Z_{1,\, n}$ behaves as expected when $\alpha_1 = 0$.

\begin{figure}[h!]
\centering
\includegraphics[width=16cm]{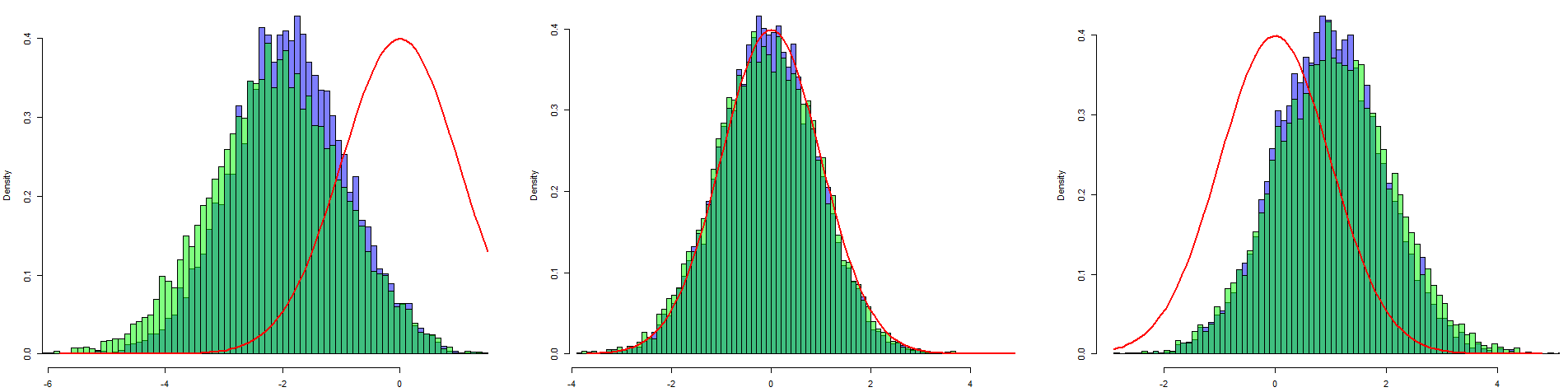}
\caption{Empirical distribution of $Z_{1,\, n}$ (blue) and $Z_{2,\, n}$ (green) for Gaussian noises $\eta_{1,\, 0} \sim \cN(0, 0.2)$ and $\veps_0 \sim \cN(0, 1)$, $\theta_1=0.3$, $\theta_2=\alpha_2=\tau_{2,\, 2} = 0$,  and $\alpha_1 \in \{ -0.5, 0, 0.3 \}$ from left to right, obtained with $N=10000$ repetitions of size $n=500$. The red curve is the theoretical $\cN(0,1)$.}
\label{FigEmpDistr}
\end{figure}

However, as it is visible on Figure \ref{FigEmpDistr} and confirmed by Figure \ref{FigRejRate}, when $\alpha_1 \neq 0$ both tests reject $\cH_0$ with a rate growing well above the 5\% threshold that we used in this experiment even if, for the same reason as before, $Z_{1,\, n}$ appears slightly more robust. This is indeed what theory predicts. Theorem \ref{ThmTLCEst} also confirms that, once correctly recentered, $Z_{1,\, n}$ and $Z_{2,\, n}$ must remain asymptotically normal.

\begin{figure}[h!]
\centering
\includegraphics[width=16cm]{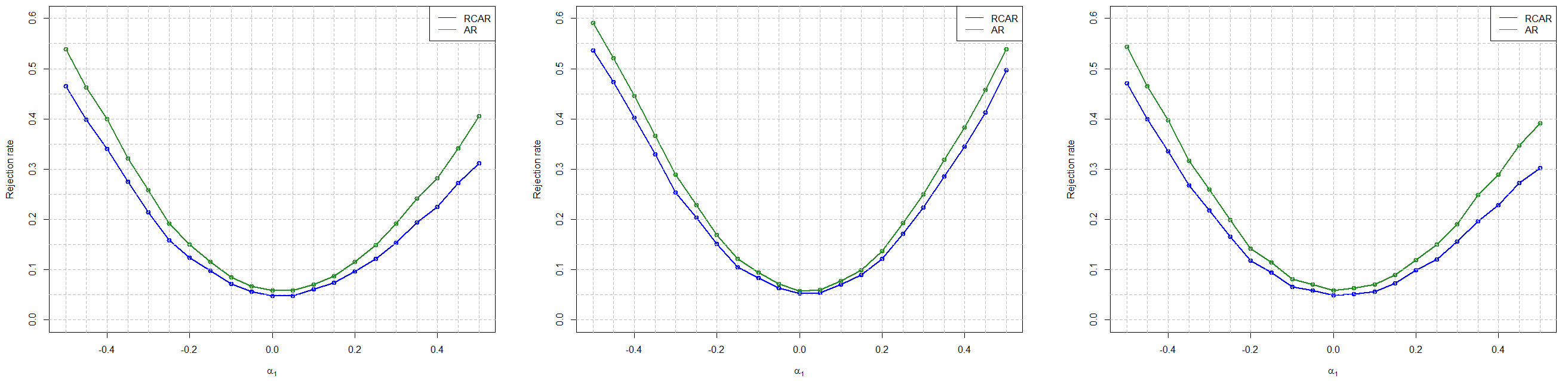}
\caption{Rate of rejection of $\cH_0$ for numerous values of $\alpha_1 \in [-0.5, 0.5]$ at the 5\% threshold. The simulations are conducted with the setting of Figure \ref{FigEmpDistr} and, from left to right, with $\theta_1 = -0.3$, $\theta_1 = 0$ and $\theta_1 = 0.3$.}
\label{FigRejRate}
\end{figure}

This example clearly highlights the question of whether we can build a consistent estimate for $\theta$ of course, but also for the covariance in the coefficients. In \cite{ProiaSoltane18}, this is done (see Sec. 4), mainly due to the fact that calculations are feasible when $p=1$, and a test for serial correlation followed. We can see through this note that this is not as easy in the general case and that this must be a reflection for future studies. Besides, we only considered the OLS but it would be worth working either with a QMLE or with a two-stage procedure to be able to exhibit a reliable estimate despite a possible serial correlation. We conclude by emphasizing the importance of developing statistical testing procedures for correlation in the random coefficients in the future works dedicated to RCAR processes. For practitioners, testing for serial correlation should appear as logic consequence of the tests for randomness in the coefficients, especially when the main hypothesis of the model is the existence of temporal correlations in the phenomenon.

\bigskip

\noindent \textbf{Acknowledgements}. The authors warmly thank the anonymous reviewer for the careful reading and for making numerous useful corrections to improve the readability of the paper. They also thank the research program PANORisk of the Région Pays de la Loire.

\section{Technical proofs}
\label{SecPro}

\subsection{Proof of Proposition \ref{PropErgo}}
\label{SecProErgo}

From our hypotheses on the noises $(\eta_{k,\, t})$, the matrix-valued process $(C_{\theta} + N_{t-1}\, D_{\alpha} + N_{t})$ is strictly stationary and ergodic. Thus, it follows from \eqref{HypLog} that we can find $\delta < 0$ and a random $k_0$ such that, as soon as $k \geq k_0$,
\begin{equation*}
\frac{1}{k} \sum_{\ell=0}^{k-1} \ln \Vert C_{\theta} + N_{t-\ell-1}\, D_{\alpha} + N_{t-\ell} \Vert < \delta \hsp \textnormal{a.s.}
\end{equation*}
See also Lem. 1.1 of \cite{Brandt86} for a similar reasoning. For $n \geq 1$, consider the truncation
\begin{equation*}
\Phi_{t,\, n} = E_{t} + \sum_{k=1}^{n} E_{t-k}\, \prod_{\ell=0}^{k-1} (C_{\theta} + N_{t-\ell-1}\, D_{\alpha} + N_{t-\ell}).
\end{equation*}
Then, by the triangle inequality, for $n$ large enough we have
\begin{equation*}
\Vert \Phi_{t,\, n} \Vert \leq \vert \veps_{t} \vert + \sum_{k=1}^{k_0-1} \vert \veps_{t-k} \vert\, \prod_{\ell=0}^{k-1} \Vert C_{\theta} + N_{t-\ell-1}\, D_{\alpha} + N_{t-\ell} \Vert + \sum_{k=k_0}^{n} \vert \veps_{t-k} \vert\, \de^{\delta k}
\end{equation*}
and Lem. 2.2 of \cite{BerkesEtAl03} ensures the a.s. convergence of the last term under \eqref{HypLog}. Thus, the limit superior of $\Vert \Phi_{t,\, n} \Vert$ is finite and so is \eqref{ExprCausal} with probability 1. Moreover, it is easy to check that this is a solution to the recurrence \eqref{RCARVec}. Finally, the strict stationarity and ergodicity of $(\Phi_{t})$ may be obtained following the same reasoning as in \cite{NichollsQuinn81b}. Indeed, one can see that there exists $\phi$ independent of $t$ such that $\Phi_{t} = \phi((E_{t}, N_{t}), (E_{t-1}, N_{t-1}), \hdots)$ and the strict stationarity and ergodicity of the process $(E_{t}, N_{t})$ are passed to $(\Phi_{t})$. That also implies the $\cF_{t}$-measurability of $(X_{t})$. \qed
 
\subsection{Proof of Theorem \ref{ThmTLCEst}}
\label{SecProTLC}

Let the filtration generated by $\cF_0^{\, *} = \sigma(\Phi_0, \eta_{1,\, 0}, \hdots, \eta_{p,\, 0})$ and, for $n \geq 1$, by
\begin{equation}
\label{FiltMart}
\cF_{n}^{\, *} = \sigma(\Phi_0, \eta_{1,\, 0}, \hdots, \eta_{p,\, 0}, (\veps_1, \eta_{1,\, 1}, \hdots, \eta_{p,\, 1}), \hdots, (\veps_n, \eta_{1,\, n}, \hdots, \eta_{p,\, n})).
\end{equation}
In the sequel, to avoid a huge amount of different notation, $M_{n}$ and $R_{n}$ will be generic terms, not necessarily identical from one line to another, designating vector $\cF_{n}^{\, *}$-martingales (see \textit{e.g.} \cite{Duflo97}) and isolated terms, respectively. We make use of the Appendix for some computational results and we start by two fundamental propositions showing that the recentered empirical covariances are $\cF_{n}^{\, *}$-martingales, up to residual terms. To this aim, we need to build the matrix $M_{\alpha, \beta}$ as follows. First, consider
\begin{equation*}
M_{\theta} = \begin{pmatrix}
\theta_1 & 0 & \cdots & \cdots & 0 \\
\theta_2 & \ddots & \ddots & & \vdots \\
\vdots & \ddots & \ddots & \ddots & \vdots \\
\vdots & & \ddots & \ddots & 0 \\
\theta_p & \cdots & \cdots & \theta_2 & \theta_1 
\end{pmatrix} + \begin{pmatrix}
0 & \theta_2 & \cdots & \cdots & \theta_p \\
\vdots & \vdots & & \udots & 0 \\
\vdots & \vdots & \udots & & \vdots \\
\vdots & \theta_p & & & \vdots \\
0 & \cdots & \cdots & \cdots & 0
\end{pmatrix}
\end{equation*}
and use it to define
\begin{equation*}
M_{\alpha, \beta} = \begin{pmatrix}
\frac{1}{1 - 2\, \beta_1} & 0 & \cdots & \cdots & 0 \\
0 & 1 & \ddots & & \vdots \\
\vdots & \ddots & \ddots & \ddots & \vdots \\
\vdots & & \ddots & \ddots & 0 \\
0 & \cdots & \cdots & 0 & 1
\end{pmatrix}\, \left[ M_{\theta} + \begin{pmatrix}
0 & \beta_2 & 0 & \cdots & 0 \\
\beta_1 & 0 & 0 & & \vdots \\
0 & \ddots & \ddots & \ddots & \vdots \\
\vdots & \ddots & \ddots & \ddots & 0 \\
0 & \cdots & 0 & \beta_1 & 0
\end{pmatrix} \right]
\end{equation*}
where $\beta_1 = \alpha_1\, \tau_{1,\, 2}$ and $\beta_2 = \alpha_2\, \tau_{2,\, 2} + \hdots + \alpha_p\, \tau_{p,\, 2}$. Then, set $U_0$ as the first column of $M_{\alpha, \beta}$ and set $K$ as the remaining part of $M_{\alpha, \beta}$ to which we add a zero vector on the right. Using the notation of the reasonings below, the pathological set $\Theta^{*}$ is enhanced with the situations where $2\, \beta_1 = 1$, $\alpha_1\, \beta_1 = 1$, $\det(I_p-K)=0$, $s_0=1$ or $U^{\, T}\, (I_p-K)^{-1}\, U_0 = 1-s_0$.

\begin{prop}
\label{PropXtXtk}
Assume that $\Theta \subset \Theta_4 \backslash \Theta^{*}$. Then, we have the decomposition
\begin{eqnarray*}
\sum_{t=1}^{n} \Phi_{t-1}\, X_{t} - n\, L_1 & = &  (I_{p} - K)^{-1}\, U_0\, \bigg( \sum_{t=1}^{n} X_{t}^{\, 2} - n\, \ell_0 \bigg) + M_{n} + R_{n}
\end{eqnarray*}
where $M_{n}$ is a vector $\cF_{n}^{\, *}$-martingale and the remainder term satisfies $\Vert R_{n} \Vert = o(\sqrt{n})$ a.s.
\end{prop}
\begin{proof}
The proposition is established through Lemma \ref{LemXtXti}. Indeed, this statement shows that, as soon as $\Theta \cap \Theta^{*} = \varnothing$, there is a decomposition of the form
\begin{eqnarray*}
\frac{1}{n} \sum_{t=1}^{n} \Phi_{t-1}\, X_{t} & = & (I_{p} - K)^{-1}\, U_0\, \frac{1}{n} \sum_{t=1}^{n} X_{t}^{\, 2} + \frac{M_{n}}{n} + \frac{R_{n}}{n}.
\end{eqnarray*}
By ergodicity, taking the limit on both sides gives $L_1 = (I_{p} - K)^{-1}\, U_0\, \ell_0$. The remainder term $R_{n}$ is made of second-order isolated terms so, by Remark \ref{RemErgo} and because $\Theta \subset \Theta_4$, we must have $\Vert R_{n} \Vert = o(\sqrt{n})$ a.s. \qed
\end{proof}

\begin{prop}
\label{PropXtsq}
Assume that $\Theta \subset \Theta_4 \backslash \Theta^{*}$. Then, we have the decomposition
\begin{eqnarray*}
\sum_{t=1}^{n} X_{t}^{\, 2} - n\, \ell_0 & = &  M_{n} + R_{n}
\end{eqnarray*}
where $M_{n}$ is a scalar $\cF_{n}^{\, *}$-martingale and the remainder term satisfies $\vert R_{n} \vert = o(\sqrt{n})$ a.s.
\end{prop}
\begin{proof}
This is a consequence of Lemma \ref{LemXtsq}. As soon as $\Theta \cap \Theta^{*} = \varnothing$, it follows from this lemma and Proposition \ref{PropXtXtk}, using the same notation, that
\begin{eqnarray*}
\bigg( \sum_{t=1}^{n} X_{t}^{\, 2} - n\, \ell_0 \bigg) \bigg[ 1-\frac{U^{\, T}\, (I_p-K)^{-1}\, U_0}{1-s_0} \bigg] & = & \frac{U^{\, T}\, M_{n} + R_{n}}{1-s_0}.
\end{eqnarray*}
This concludes the proof since, likewise, the remainder term $R_{n}$ is a linear combination of second-order isolated terms. \qed
\end{proof}

We now come back to the proof of Theorem \ref{ThmTLCEst}. The keystone of the reasoning consists in noting that there is a matrix $A$ such that
\begin{equation*}
\sum_{t=1}^{n} \Phi_{t-1}\, X_{t} - S_{n-1}\, \theta^{\, *} = A\, \bigg( \sum_{t=1}^{n} \Phi_{t-1}\, X_{t} - n\, L_1 \bigg) - \theta^{\, *}\, \bigg( \sum_{t=1}^{n} X_{t}^{\, 2} - n\, \ell_0 \bigg) + R_{n}
\end{equation*}
since, by ergodicity, $0 = L_1 - \Lambda_0\, \theta^{\, *} = A\, L_1 - \theta^{\, *}\, \ell_0$. The combination of Propositions \ref{PropXtXtk} and \ref{PropXtsq} enables to obtain the decomposition
\begin{equation}
\label{DecompTlc}
\sum_{t=1}^{n} \Phi_{t-1}\, X_{t} - S_{n-1}\, \theta^{\, *} = M_{n} + R_{n}
\end{equation}
where, as in the previous proofs, $M_{n}$ is a vector $\cF_{n}^{\, *}$-martingale and the remainder term satisfies $\Vert R_{n} \Vert = o(\sqrt{n})$ a.s. Let us call $m_{n}$ a generic element of $M_{n}$. As can be seen from the details of the Appendix, it always takes the form of
\begin{equation}
\label{GenMart}
m_{n} = \sum_{t=1}^{n} X_{t-d_1}^{\, a_1}\, X_{t-d_2}^{\, a_2}\, \eta_{k,\, t-d_3}^{\, a_3}\, \eta_{\ell,\, t-d_4}^{\, a_4}\, \nu_{t}
\end{equation}
where $0 < d_3, d_4 \leq d_1, d_2$, $a_i \in \{ 0, 1 \}$ and where the zero-mean random variable $\nu_{t}$ is identically distributed and independent of $\cF_{t-1}^{\, *}$. Provided that $\sigma_4 < +\infty$ and $\tau_{k,\, 8} < +\infty$ (included in $\Theta_4$), the ergodicity arguments \eqref{Ergo} together with the fact that $\Theta \subset \Theta_4$ show that there exists a matrix $L$ satisfying $\langle M \rangle_{n}/n \rightarrow L$ a.s. From the causal representation of the process given in Proposition \ref{PropErgo} and the hypothesis on $\Phi_0$, the increments of $m_{n}$ are also strictly stationary and ergodic. Thus, for any $x > 0$,
$$
\frac{1}{n} \sum_{t=1}^{n} \dE[ (\Delta m_{t})^{\, 2}\, \ind_{ \{ \vert \Delta m_{t} \vert\, \geq\, x \}}\, \vert\, \cF_{t-1}^{\, *}] \cvgas \dE[ (\Delta m_1)^{\, 2}\, \ind_{ \{ \vert \Delta m_1 \vert\, \geq\, x \}}].
$$
Since $\dE[(\Delta m_1)^{\, 2}] < +\infty$, the right-hand side can be made arbitrarily small for $x \rightarrow+\infty$. Once again generalizing to $M_{n}$, we obtain \textit{via} the same arguments that the Lindeberg's condition is satisfied. We are now ready to apply the central limit theorem for vector martingales, given e.g. by Cor. 2.1.10 of \cite{Duflo97}, and get that $M_{n}/\sqrt{n}$ is asymptotically normal with mean 0 and covariance $L$. Together with Remark \ref{RemErgo}, that leads to
\begin{equation*}
\frac{1}{\sqrt{n}}\, \bigg( \sum_{t=1}^{n} \Phi_{t-1}\, X_{t} - S_{n-1}\, \theta^{\, *} \bigg) \cvgl \cN(0,\, L).
\end{equation*}
Finally, the a.s. convergence of $n\, S_{n-1}^{\, -1}$ to $\Lambda_0^{\, -1}$ and Slutsky's lemma conclude the proof. \qed

\subsection{Proof of Theorem \ref{ThmRatEst}}
\label{SecProRat}

Let us come back to the $\cF_{n}^{\, *}$-martingale $m_{n}$ given in \eqref{GenMart}. By the Hartman-Wintner law of the iterated logarithm for martingales (see \textit{e.g.} \cite{Stout70}),
\begin{equation*}
\limsup_{n\, \rightarrow\, +\infty} ~ \frac{m_{n}}{\sqrt{2\, n\, \ln \ln n}} = v_{m} \hsp \text{and} \hsp \liminf_{n\, \rightarrow\, +\infty} ~ \frac{m_{n}}{\sqrt{2\, n\, \ln \ln n}} = -v_{m} \hsp \textnormal{a.s.}
\end{equation*}
where $v_{m} = \dE[(\Delta m_1)^{\, 2}] < +\infty$, provided that $\Theta \subset \Theta_4$. The limit inferior was obtained by replacing $m_{n}$ by $-m_{n}$, which share the same variance and martingale properties. Exploiting the latter bounds and generalizing to $M_{n}$, we can deduce that
\begin{equation}
\label{RatMart}
\limsup_{n\, \rightarrow\, +\infty} ~ \frac{\Vert M_{n} \Vert}{\sqrt{2\, n\, \ln \ln n}} < +\infty \hsp \textnormal{a.s.}
\end{equation}
Because \eqref{DecompTlc} implies $\Vert \wht - \theta^{\, *} \Vert \leq \Vert S_{n-1}^{\, -1} \Vert\, \Vert M_{n} + R_{n} \Vert$, it is now easy to conclude the proof \textit{via} \eqref{RatMart}. Indeed, we recall that $n\, S_{n-1}^{\, -1}$ is convergent and that $\Vert R_{n} \Vert = o(\sqrt{n})$, obviously implying that it is also $o(\sqrt{n\, \ln \ln n})$ a.s. \qed

\appendix
\renewcommand{\thesection}{\Alph{section}}
\section{Some martingales decompositions}
\label{SecAppDecomp}

For an easier reading, let us gather in this Appendix the proofs only based on calculations. Like in Section \ref{SecPro}, $M_{n}$, $\delta_{n}$ and $R_{n}$ will be generic terms, not necessarily identical from one line to another, designating $\cF_{n}^{\, *}$-martingales, differences of $\cF_{n}^{\, *}$-martingales and isolated terms, respectively, where $\cF_{n}^{\, *}$ is the filtration defined in \eqref{FiltMart}. To save place, we deliberately skip the proofs for most of them because they only consist of tedious but straightforward calculations. We assume in all the Appendix that $\Theta \subset \Theta_4 \backslash \Theta^{*}$.

\begin{lem}
\label{LemXtsqN1t}
We have the decomposition
\begin{eqnarray*}
\sum_{t=1}^{n} X_{t}^{\, 2}\, \eta_{1,\, t} & = & \frac{2\, \tau_{1,\, 2}}{1 - 2\, \beta_1} \Bigg[ \sum_{k=1}^{p} \theta_k \sum_{t=1}^{n} X_{t}\, X_{t-k+1} + \beta_2 \sum_{t=1}^{n} X_{t}\, X_{t-1} \Bigg] + M_{n} + R_{n}
\end{eqnarray*}
where $\beta_1 = \alpha_1\, \tau_{1,\, 2}$ and $\beta_2 = \alpha_2\, \tau_{2,\, 2} + \hdots + \alpha_p\, \tau_{p,\, 2}$.
\end{lem}
\begin{proof}
Develop $X_{t}$ using \eqref{RCAR} to get that, for all $t$ and after a lot of simplifications,
\begin{equation*}
X_{t}^{\, 2}\, \eta_{1,\, t} = 2\, \tau_{1,\, 2}\, \bigg( \sum_{k=1}^{p} \theta_k\, X_{t-1}\, X_{t-k} + \sum_{k=2}^{p} \alpha_k\, \tau_{k,\, 2}\, X_{t-k}\, X_{t-k-1} + \alpha_1\, X_{t-1}^{\, 2}\, \eta_{1,\, t-1} \bigg) + \delta_{t}.
\end{equation*}
It remains to sum over $t$ and to gather all equivalent terms. Note that $\{ 2\, \beta_1 = 1 \} \subset \Theta^{*}$.
\end{proof}

\begin{lem}
\label{LemXtXti}
We have the decomposition
\begin{eqnarray*}
\sum_{t=1}^{n} X_{t}\, X_{t-1} & = & \frac{1}{1 - 2\, \beta_1} \Bigg[ \sum_{k=1}^{p} \theta_k \sum_{t=1}^{n} X_{t}\, X_{t-k+1} + \beta_2 \sum_{t=1}^{n} X_{t}\, X_{t-1} \Bigg] + M_{n} + R_{n}
\end{eqnarray*}
where $\beta_1 = \alpha_1\, \tau_{1,\, 2}$ and $\beta_2 = \alpha_2\, \tau_{2,\, 2} + \hdots + \alpha_p\, \tau_{p,\, 2}$. In addition, for $i \geq 2$, we have
\begin{eqnarray*}
\sum_{t=1}^{n} X_{t}\, X_{t-i} & = & \sum_{k=1}^{p} \theta_k \sum_{t=1}^{n} X_{t}\, X_{t-\vert i - k \vert} + \beta_1 \sum_{t=1}^{n} X_{t}\, X_{t-\vert i - 2 \vert} + M_{n} + R_{n}.
\end{eqnarray*}
\end{lem}
\begin{proof}
Develop $X_{t}$ using \eqref{RCAR} to get, after additional calculations, that for all $t$,
\begin{equation*}
X_{t}\, X_{t-i} = \sum_{k=1}^{p} \theta_k\, X_{t-i}\, X_{t-k} + \alpha_1\, \tau_{1,\, 2}\, X_{t-2}\, X_{t-i} + \delta_{t}
\end{equation*}
as soon as $i \geq 2$. For $i=1$, we can show that for all $t$,
\begin{equation*}
X_{t}\, X_{t-1} = \sum_{k=1}^{p} \theta_k\, X_{t-1}\, X_{t-k} + \alpha_1\, X_{t-1}^{\, 2}\, \eta_{1,\, t-1} + \sum_{k=2}^{p} \alpha_k\, \tau_{k,\, 2}\, X_{t-k}\, X_{t-k-1} + \delta_{t}.
\end{equation*}
Finally, it remains to sum over $t$ and to exploit Lemma \ref{LemXtsqN1t}. Note that $\{ 2\, \beta_1 = 1 \} \subset \Theta^{*}$.
\end{proof}

\begin{lem}
\label{LemXtsqN1tsq}
There exists $\gamma_0, \hdots, \gamma_{p-1}$ such that we have the decomposition
\begin{eqnarray*}
\sum_{t=1}^{n} X_{t}^{\, 2}\, \eta_{1,\, t}^{\, 2} & = & \frac{1}{1 - \alpha_1\, \beta_1}\, \Bigg[ \sum_{k=1}^{p} \gamma_{k-1} \sum_{t=1}^{n} X_{t}\, X_{t-k+1} + \sigma_2\, \tau_{1,\, 2}\, n \Bigg] + M_{n} + R_{n}
\end{eqnarray*}
where $\beta_1 = \alpha_1\, \tau_{1,\, 2}$.
\end{lem}
\begin{proof}
We first obtain, using \eqref{RCAR}, that for all $t$,
\begin{equation*}
X_{t}^{\, 2}\, \eta_{1,\, t}^{\, 2} = X_{t}\, \eta_{1,\, t}^{\, 2}\, \Bigg( \sum_{k=1}^{p} \theta_k\, X_{t-k} + \sum_{k=1}^{p} X_{t-k}\, \eta_{k,\, t} + \sum_{k=1}^{p} \alpha_k\, X_{t-k}\, \eta_{k,\, t-1} \Bigg)   + \eta_{1,\, t}^{\, 2}\, \veps_{t}^{\, 2} + \delta_{t}.
\end{equation*}
Then, the statements above lead to the result. Note that $\{ \alpha_1\, \beta_1 = 1 \} \subset \Theta^{*}$.
\end{proof}

\begin{lem}
\label{LemXtsq}
There exists $U^{\, T} = (u_1, \hdots, u_{p-1}, 0)$ such that we have the decomposition
\begin{eqnarray*}
\sum_{t=1}^{n} X_{t}^{\, 2} & = & \frac{U^{\, T}}{1-s_0}\, \sum_{t=1}^{n} \Phi_{t-1}\, X_{t} + \bigg( \ell_0 - \frac{U^{\, T}\, L_1}{1-s_0} \bigg)\, n + M_{n} + R_{n}
\end{eqnarray*}
where
\begin{equation*}
s_0 = \bigg( 1 + \alpha_1\, \theta_2 + \frac{2\, \alpha_1\, \theta_1^{\, 2}}{1 - 2\, \beta_1} + \frac{\sigma_2}{1 - \alpha_1\, \beta_1} \bigg)\, \tau_{1,\, 2} + \sum_{k=2}^{p} (1 + \alpha_k^{\, 2})\, \tau_{k,\, 2}
\end{equation*}
and $\beta_1 = \alpha_1\, \tau_{1,\, 2}$.
\end{lem}
\begin{proof}
As in the previous proofs, we start by developping $X_{t}$ using \eqref{RCAR} to get, for all $t$,
\begin{equation*}
X_{t}^{\, 2} = X_{t}\, \Bigg( \sum_{k=1}^{p} \theta_k\, X_{t-k} + \sum_{k=1}^{p} X_{t-k}\, \eta_{k,\, t} + \sum_{k=1}^{p} \alpha_k\, X_{t-k}\, \eta_{k,\, t-1} \Bigg) + \veps_{t}^{\, 2} + \delta_{t}.
\end{equation*}
Finally, after some additional steps of calculations, Lemmas \ref{LemXtsqN1t} and \ref{LemXtsqN1tsq} enable to simplify the summation. Note that $\{ 2\, \beta_1 = 1 \} \cup \{ \alpha_1\, \beta_1 = 1 \} \cup \{ s_0 = 1 \} \subset \Theta^{*}$.
\end{proof}

\section{Moments of the process}
\label{SecAppMom}

In this section, we repeatedly need the well-known relations $\vec(A X B) = (B^{\, T} \otimes A)\, \vec(X)$ and $(A \otimes B)(C \otimes D) = (AC) \otimes (BD)$. First, let us give more details on the condition $\dE[\eta_{k,\, t}^{\, a}\, X_{t}^{\, 2}] < +\infty$ for $a \in \{ 0, 1, 2 \}$. From the vector form \eqref{RCARVec} of the process together with its $\cF_{t}$-measurability \eqref{ExprCausal}, we may obtain that for all $t$, developing and taking expectation,
\begin{eqnarray*}
\dE[\vec(\Phi_{t}\, \Phi_{t}^{\, T})] & = & (C_{\theta} \otimes C_{\theta} + \Gamma_0)\, \dE[\vec(\Phi_{t-1}\, \Phi_{t-1}^{\, T})] \\
 & & \hsp +~ (I_p \otimes C_{\theta})\, \dE[((N_{t-1}\, D_{\alpha}) \otimes I_p)\, \vec(\Phi_{t-1}\, \Phi_{t-1}^{\, T})] \\
 & & \hsp +~ (C_{\theta} \otimes I_p)\, \dE[(I_p \otimes (N_{t-1}\, D_{\alpha}))\, \vec(\Phi_{t-1}\, \Phi_{t-1}^{\, T})] \\
 & & \hsp +~ \dE[((N_{t-1}\, D_{\alpha}) \otimes (N_{t-1}\, D_{\alpha}))\, \vec(\Phi_{t-1}\, \Phi_{t-1}^{\, T})] + \vec(\Sigma_0)
\end{eqnarray*}
where $\Gamma_0 = \dE[N_0 \otimes N_0]$ and $\Sigma_0 = \dE[E_0\, E_0^{\, T}]$. In a more compact form,
\begin{equation*}
U_{t} = (C_{\theta} \otimes C_{\theta} + \Gamma_0)\, U_{t-1} + (I_p \otimes C_{\theta})\, V_{1,\, t-1} + (C_{\theta} \otimes I_p)\, V_{2,\, t-1} + W_{t-1} + \vec(\Sigma_0)
\end{equation*}
where $U_{t}$, $V_{1,\, t}$, $V_{2,\, t}$ and $W_{t}$ are easily identifiable from the explicit relation above. Working similarly on the other components and stacking them into $\Omega_{t}$, we get the linear system
\begin{equation}
\label{EqLinMom2}
\Omega_{t} = A_2\, \Omega_{t-1} + B_2
\end{equation}
where $B_2^{\, T} = (\vec(\Sigma_0), 0, 0, \Gamma_{\alpha \alpha}\, \vec(\Sigma_0))$,
\begin{equation*}
A_2 = \begin{pmatrix}
C_{\theta} \otimes C_{\theta} + \Gamma_0 & I_p \otimes C_{\theta} & C_{\theta} \otimes I_p & I_{p^2} \\
G_{\alpha} \otimes C_{\theta} + \Gamma_{\alpha c} & \Gamma_{\alpha} & G_{\alpha} \otimes I_p & 0 \\
C_{\theta} \otimes G_{\alpha} + \Gamma_{\alpha c}^{\, \prime} & I_p \otimes G_{\alpha} & \Gamma_{\alpha}^{\, \prime} & 0 \\
\Gamma_{\alpha \alpha}\, (C_{\theta} \otimes C_{\theta}) + \Lambda_{\alpha \alpha} & \Gamma_{\alpha \alpha}\, (I_p \otimes C_{\theta}) & \Gamma_{\alpha \alpha}\, (C_{\theta} \otimes I_p) & \Gamma_{\alpha \alpha}
\end{pmatrix}
\end{equation*}
and where, additionally, $G_{\alpha} = \dE[N_0\, D_{\alpha}\, N_0]$, $\Gamma_{\alpha} = \dE[(N_{0}\, D_{\alpha}) \otimes N_{0}]$, $\Gamma_{\alpha}^{\, \prime} = \dE[N_0 \otimes (N_{0}\, D_{\alpha})]$, $\Gamma_{\alpha c} = \dE[(N_0\, D_{\alpha}\, C_{\theta}) \otimes N_0]$, $\Gamma_{\alpha c}^{\, \prime} = \dE[N_0 \otimes (N_0\, D_{\alpha}\, C_{\theta})]$, $\Gamma_{\alpha \alpha} = \dE[(N_{0}\, D_{\alpha}) \otimes (N_{0}\, D_{\alpha})]$ and $\Lambda_{\alpha \alpha} = \dE[(N_{0}\, D_{\alpha}\, N_{0}) \otimes (N_{0}\, D_{\alpha}\, N_{0})]$. In virtue of \eqref{EqLinMom2}, the condition $\rho(A_2) < 1$ is necessary and sufficient for the existence of the second-order moments of the process, like in \cite{NichollsQuinn81a}. Note that the fourth-order moments of $(\eta_{k,\, t})$ and the second-order moments of $(\veps_{t})$ are involved.

\begin{rem}
\label{RemEstParam}
Let $\Omega = (I_{4p^2}-A_2)^{-1} B_2$ be the steady state of \eqref{EqLinMom2} with vertical blocks $U$, $V_1$, $V_2$ and $W$. Then, by ergodicity,
\begin{equation*}
\frac{1}{n} \sum_{t=1}^{n} \Phi_{t-1}\, \Phi_{t-1}^{\, T} \cvgas \vec^{-1}(U) \hsp \text{and} \hsp \frac{1}{n} \sum_{t=1}^{n} \Phi_{t-1}\, \Phi_{t-1}^{\, T}\, D_{\alpha}\, N_{t-1} \cvgas \vec^{-1}(V_1)
\end{equation*}
where $\vec^{-1} : \dR^{p^2} \rightarrow \dR^{p \times p}$ stands for the inverse operator of $\vec : \dR^{p \times p} \rightarrow \dR^{p^2}$. Thus, we deduce from \eqref{RCARVec} and \eqref{OLS} that $\theta^{\, *} = \theta + \Lambda_0^{-1} \lambda$, where $\Lambda_0 = \vec^{-1}(U)$, given in \eqref{Lam}, and $\lambda$ is the first column of $\vec^{-1}(V_1)$. In particular, $\theta^{\, *} = \theta \Leftrightarrow \lambda=0 \Leftarrow \alpha=0$.
\end{rem}
The condition $\dE[\eta_{k,\, t}^{\, a}\, X_{t}^{\, 4}] < +\infty$ for $a \in \{ 0, \hdots, 4 \}$ can also be treated using the same strategy. However, due to the extent of calculations, we cannot afford it in this paper and we just give an outline. Instead of working on $\dE[\vec(\Phi_{t}\, \Phi_{t}^{\, T})]$ like in the case of the second-order moments, we have to start with the treatment of $\dE[\vec(\vec(\Phi_{t}\, \Phi_{t}^{\, T})\, \vec^{T}(\Phi_{t}\, \Phi_{t}^{\, T}))]$ which shall lead, after an extremely long development, to another system of the form
\begin{equation}
\label{EqLinMom4}
\Pi_{t} = A_4\, \Pi_{t-1} + B_4.
\end{equation}
Here, the fourth-order moments of $(\veps_{t})$ and the eighth-order moments of $(\eta_{k,\, t})$ are involved. Finally, the condition of existence of the fourth-order moments of the process is $\rho(A_4) < 1$, as it is done in \cite{ProiaSoltane18} for $p=1$.

\nocite{*}

\bibliographystyle{apalike} 
\bibliography{RCARMACoefGen}

\vspace{10pt}

\end{document}